\documentclass[12pt]{amsart}
\usepackage{graphicx,amscd,color}
\usepackage{a4wide}

\theoremstyle{plain}
\newtheorem{theorem}{Theorem}[section]

\newtheorem{lemma}[theorem]{Lemma}
\newtheorem{corollary}[theorem]{Corollary}

\theoremstyle{definition}

\newcommand{\N}{\operatorname{N}}

\theoremstyle{remark}

\makeatletter
\@namedef{subjclassname@2020}{\textup{2020} Mathematics Subject Classification}
\makeatother

\begin{document}

\title[The primitive curve complex for a handlebody]{The primitive curve complex for a handlebody}

\author[S. Cho]{Sangbum Cho}
\thanks{The first-named author is supported by the National Research Foundation of Korea(NRF) grant funded by the Korea government(MSIT) (NRF-2021R1F1A1060603).}
\address{Department of Mathematics Education,
Hanyang University, Seoul 04763, Korea, and School of Mathematics,
Korea Institute for Advanced Study, Seoul 02455, Korea}
\email{scho@hanyang.ac.kr}

\author[J. H. Lee]{Jung Hoon Lee}
\thanks{The second-named author is supported by the National Research Foundation of Korea(NRF) grant funded by the Korea government(MSIT) (RS-2023-00275419).}
\address{Department of Mathematics and Institute of Pure and Applied Mathematics,
Jeonbuk National University, Jeonju 54896, Korea}
\email{junghoon@jbnu.ac.kr}



\begin{abstract}
A simple closed curve in the boundary surface of a handlebody is called primitive if there exists an essential disk in the handlebody whose boundary circle intersects the curve transversely in a single point.
The primitive curve complex is then defined to be the full subcomplex of the curve complex for the boundary surface, spanned by the vertices of primitive curves.
Given any two primitive curves, we construct a sequence of primitive curves from one to the other one satisfying a certain property.
As a consequence, we prove that the primitive curve complex for the handlebody is connected.
\end{abstract}

\maketitle

\section{Introduction}\label{sec1}

Let $V$ be a genus-$g$ handlebody for $g \geq 2$, and let $\Sigma$ be the boundary of $V$, a closed orientable surface of the same genus.
A simple closed curve $C$ in $\Sigma = \partial V$ is called a {\em primitive curve} if there exists an essential disk $D$ in $V$ such that $C$ intersects $\partial D$ transversely in a single point.
Such a disk $D$ is called a {\em dual disk} of $C$.
Of course any primitive curve in $\Sigma$ admits infinitely many non-isotopic dual disks, and conversely any non-separating essential disk in $V$ can be a dual disk of infinitely many non-isotopic primitive curves.

Two primitive curves $C$ and $C'$ are said to be {\em separated} if there exist dual disks $D$ and $D'$ of $C$ and $C'$ respectively such that $C \cup D$ and $C' \cup D'$ are disjoint.
If two primitive curves $C$ and $C'$ are separated with their dual disks $D$ and $D'$, then one can find quickly their common dual disk by taking a ``band sum'' of $D$ and $D'$.
On the other hand, there are infinitely many of disjoint primitive curves $C$ and $C'$ that are not separated although they have a common dual disk (even though $C \cup C'$ is nonseparating in $\Sigma$).
The main result of this work is stated as follows.

\begin{theorem}\label{thm1.1}
Let $C$ and $C'$ be primitive curves in $\Sigma$, the boundary of a genus-$g$ handlebody $V$ for $g \ge 2$.
Then there exists a sequence $C = C_1, C_2, \ldots, C_n = C'$ of primitive curves in $\Sigma$ such that $C_i$ and $C_{i+1}$ are separated for each $i \in \{ 1, 2, \ldots, n-1 \}$.
\end{theorem}

We first prove a weaker version of Theorem \ref{thm1.1}, in which $C_i$ and $C_{i+1}$ are only required to be disjoint and to have a common dual disk for each $i \in \{ 1, 2, \ldots, n-1 \}$ (Theorem \ref{thm2.4}), and then Theorem \ref{thm1.1} will be proved for the case of $g=2$ (Theorem \ref{thm3.1}) and for the case of $g \ge 3$ (Theorem \ref{thm4.4}).
In any case, the dual disks of $C$ and $C'$ can be chosen arbitrarily in the first line of the proof.

The curve complex $\mathcal C(\Sigma)$ for a closed orientable surface $\Sigma$ of genus $g$ with $g \geq 2$ is a simplicial complex defined as follows.
The vertices of $\mathcal C(\Sigma)$ are the isotopy classes of essential simple closed curves in $\Sigma$, and a collection of $k+1$ distinct vertices of $\mathcal C(\Sigma)$ spans a $k$-simplex if it admits a collection of representatives, all of which are pairwise disjoint.
The combinatorial structure of $\mathcal C(\Sigma)$ has been widely studied for the past decades.
In particular, it is well known that $\mathcal C(\Sigma)$ is a $(3g - 4)$-dimensional connected complex, and in fact it is homotopy equivalent to a wedge sum of spheres.
In this work, we are interested in a special subcomplex of $\mathcal C(\Sigma)$ when $\Sigma$ bounds a handlebody $V$, called the primitive curve complex.

The {\em primitive curve complex} $\mathcal{PC}(V)$ for $V$ is the full subcomplex of $\mathcal C(\Sigma)$ spanned by vertices whose representatives are primitive curves in $\Sigma$.
The following is a direct consequence of Theorem \ref{thm1.1} (it is enough to see that $C_i$ and $C_{i+1}$ are disjoint in the theorem).

\begin{corollary}\label{cor1.2}
The primitive curve complex $\mathcal{PC}(V)$ for a genus-$g$ handlebody $V$ is connected for every $g \ge 2$.
\end{corollary}

Understanding combinatorial structure of a simplicial complex like $\mathcal{PC}(V)$ is not only an interesting problem itself, but also it makes many useful applications.
For example, Wajnryb \cite{W} showed that “the cut-system complex” for a handlebody is simply connected, and found a finite presentation of the handlebody group by investigating the action of the group on the complex.
The handlebody group also acts on the primitive curve complex $\mathcal{PC}(V)$ in a natural way, and so $\mathcal{PC}(V)$ would be an alternate to find a (simpler) presentation of the group.

We are also interested in a subcomplex of $\mathcal{PC}(V)$, called the primitive disk complex, when the handlebody $V$ is standardly embedded in the $3$-sphere.
That is, when the closure $W$ of the complement of $V$ in the $3$-sphere is also a handlebody, which forms a Heegaard splitting of the $3$-sphere together with the handlebody $V$.
The {\em primtive disk complex} $\mathcal P(W)$ is then defined to be the full subcomplex of $\mathcal{PC}(V)$ spanned by the vertices of $\mathcal{PC}(V)$ whose representatives bound disks in the handlebody $W$.
If $\mathcal P(W)$ is a connected complex and moreover if it is also true a version of Theorem \ref{thm1.1} for the primitive disks in $W$, then one can show quickly that the reducing sphere complex for the splitting is also connected, which implies that the Powell conjecture is true \cite{Z}.

Throughout the paper, $\overline{X}$ will denote the closure of $X$ and $\N(X)$ a regular neighborhood of $X$ for a subspace $X$ of a space, where the ambient space will always be clear from the context.

\section{Primitive curves having a common dual disk}\label{sec2}

The goal of this section is to prove Theorem \ref{thm2.4}, a weaker version of our main theorem.
To avoid repeating the expression of ``sequence of curves'' in the arguments, we simply say that two primitive curves $C$ and $C'$ {\em are c-connected} if there exists a sequence $C = C_1, C_2, \ldots, C_k = C'$ of primitive curves for some $k$ such that $C_i$ and $C_{i+1}$ are disjoint and have a common dual disk for each $i \in \{ 1, 2, \ldots, k-1 \}$.
If all of $C = C_1, C_2, \ldots, C_k = C'$ admit a single common dual disk $D$, then we say that $C$ and $C'$ {\em are c-connected with the common dual disk $D$}.

\begin{lemma}[Intersecting in a single point]\label{lem2.1}
Let $C$ and $C'$ be primitive curves in $\Sigma$, the boundary of a genus-$g$ handlebody $V$ for $g \ge 2$.
Suppose that $C$ and $C'$ have a common dual disk $D$ and that $| C \cap C' | = 1$.
Then $C$ and $C'$ are c-connected with the common dual disk $D$.
\end{lemma}

\begin{proof}
Cutting $V$ along $D$, we have a genus-$(g-1)$ handlebody $V'$ with two copies $D_{+}$ and $D_{-}$ of $D$ on $\Sigma' = \partial V'$.
The primitive curve $C$ is then cut into an arc $\alpha$ with one of its endpoints $c_{+}$ in $\partial D_{+}$ and the other endpoint $c_{-}$ in $\partial D_{-}$.
The primitive curve $C'$ is also cut into an arc $\beta$ with its endpoints $c'_{+}$ and $c'_{-}$ in $\partial D_{+}$ and $\partial D_{-}$ respectively.
Let $p = \alpha \cap \beta$.
Let $\alpha_{+}$ and $\alpha_{-}$ be the subarcs of $\alpha$ with endpoints $\{ c_{+}, p \}$ and $\{ p, c_{-} \}$ respectively.
Let $\beta_{+}$ and $\beta_{-}$ be the subarcs of $\beta$ with endpoints $\{ c'_{+}, p \}$ and $\{ p, c'_{-} \}$ respectively.
First we consider the following special case.

\vspace{0.3cm}

\noindent Case $1$. Both $\alpha_{+} \cup \beta_{+}$ and $\alpha_{-} \cup \beta_{-}$ cut off disks $\Delta_{+}$ and $\Delta_{-}$ from $\overline{\Sigma' - D_{+}}$ and $\overline{\Sigma' - D_{-}}$ respectively.
There are two subcases (see Figure \ref{fig1} (a) and (b)).

\begin{enumerate}
\item[(a)] The two disks $\Delta_{+}$ and $\Delta_{-}$ intersect only in the point $p$.
\item[(b)] One of $\Delta_{+}$ and $\Delta_{-}$, say $\Delta_{-}$, contains $\Delta_{+}$.
\end{enumerate}

\begin{figure}[!hbt]
\includegraphics[width=14cm,clip]{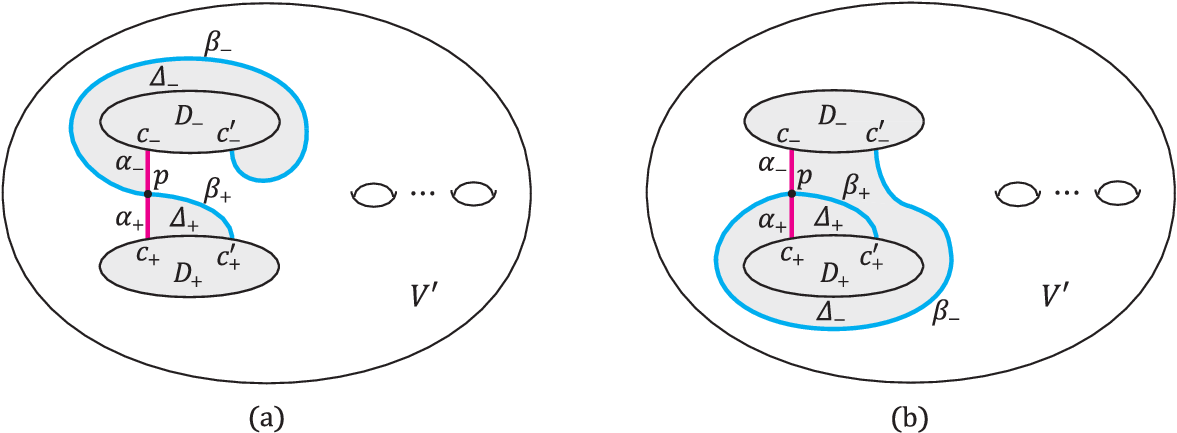}
\caption{(a) $\Delta_{+} \cap \Delta_{-} = p$ and (b) $\Delta_{+} \subset \Delta_{-}$.}
\label{fig1}
\end{figure}

\noindent (a) Take a point $d_{+}$ in $\partial D_{+} \cap \Delta_{+}$ between $c_{+}$ and $c'_{+}$.
Let $d_{-}$ be the point in $\partial D_{-}$ that is identified with $d_{+}$.
Then $d_{-}$ is not in $\Delta_{-}$.
Take a point $q$ in $\alpha_{+}$ between $c_{+}$ and $p$.
Let $\delta_{+}$ be an arc in $\Delta_{+}$ connecting $d_{+}$ and $q$.
Since $D_{+} \cup D_{-} \cup \Delta_{+} \cup \Delta_{-}$ is homotopy equivalent to a point, we can take a non-separating arc $\delta_{-}$ in $\overline{\Sigma' - (D_{+} \cup D_{-} \cup \Delta_{+} \cup \Delta_{-})}$ connecting $q$ and $d_{-}$ (see Figure \ref{fig2}).

\begin{figure}[!hbt]
\includegraphics[width=7cm,clip]{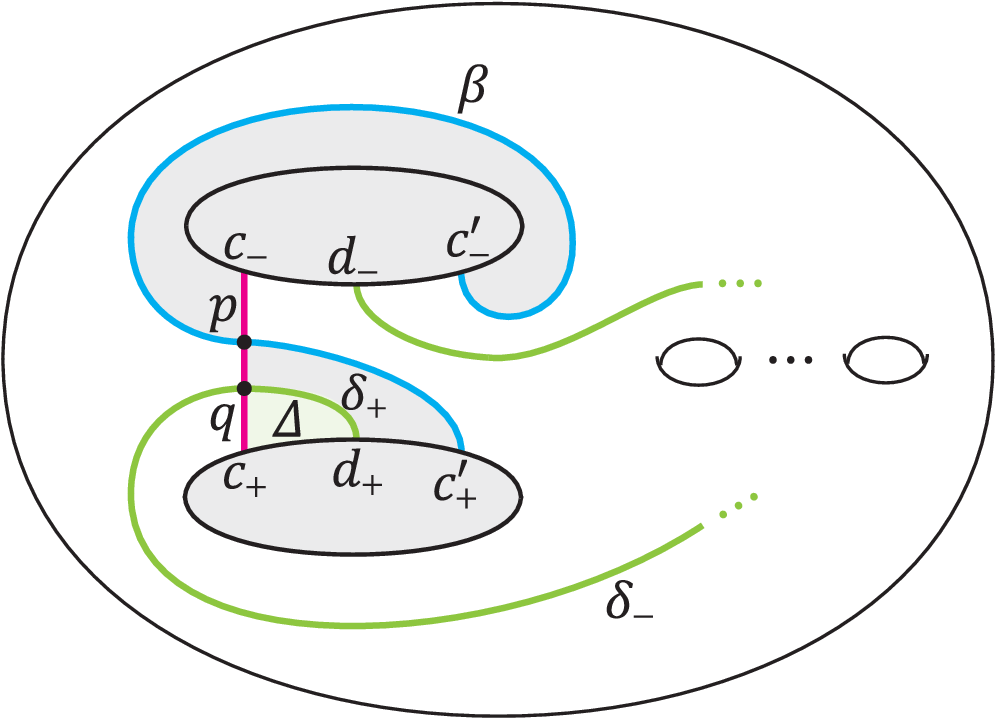}
\caption{The arc $\delta_{+}$ and $\delta_{-}$.}
\label{fig2}
\end{figure}

Let $\delta = \delta_{+} \cup \delta_{-}$.
If we identify $D_{+}$ and $D_{-}$, then $\delta$ becomes a primitive curve with a dual disk $D$.
The arc $\delta$ is disjoint from $\beta$, and $\delta$ intersects $\alpha$ in a single point $q$.
Let $\Delta$ be the triangular disk determined by $\{ q, c_{+}, d_{+} \}$.
Since $\delta_{-}$ is non-separating in $\overline{\Sigma' - (D_{+} \cup \Delta \cup \N{(\alpha)} \cup D_{-})}$, we can take an arc $\epsilon$ in $\overline{\Sigma' - (D_{+} \cup \Delta \cup \N{(\alpha)} \cup D_{-})}$ with endpoints $e_{+}$ and $e_{-}$ in $\partial D_{+}$ and $\partial D_{-}$ respectively with the following properties (see Figure \ref{fig3}).

\begin{itemize}
\item The arc $\epsilon$ is disjoint from $\alpha$ and $\delta$.
\item The two points $e_{+}$ and $e_{-}$ become the same point if we identify $D_{+}$ and $D_{-}$.
\end{itemize}

Then $\delta$ is c-connected to $\alpha$ via $\epsilon$ (if we regard the endpoints of each arc are identified in $V$) with a common dual disk $D$.

\begin{figure}[!hbt]
\includegraphics[width=7cm,clip]{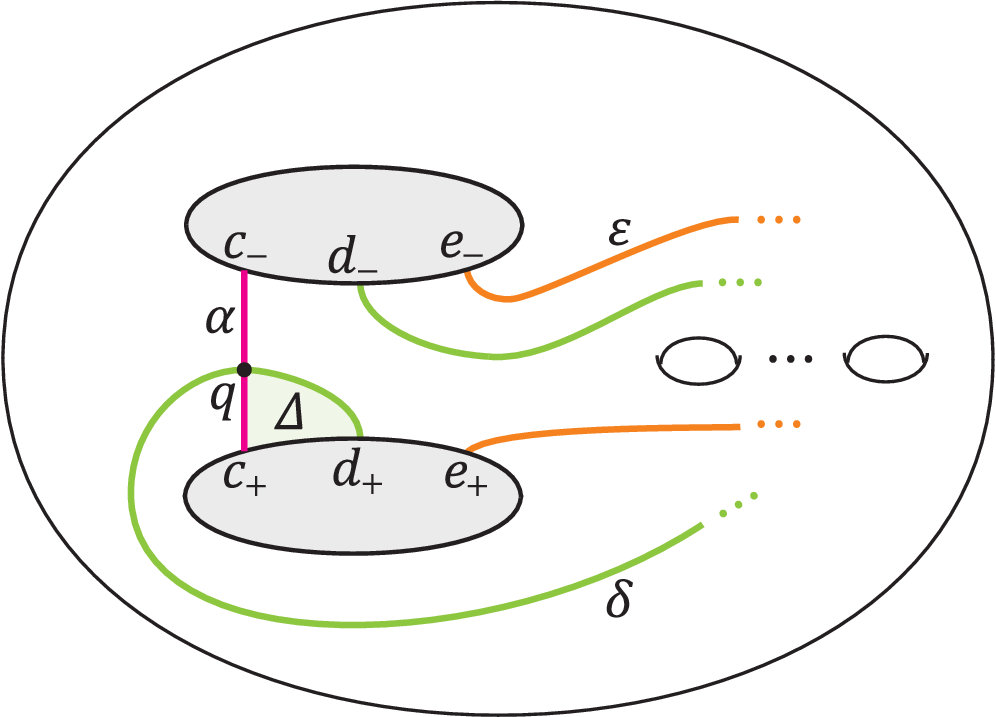}
\caption{The arc $\epsilon$.}
\label{fig3}
\end{figure}

\vspace{0.3cm}

\noindent (b) Take a point $d_{+}$ in $\partial D_{+} - \Delta_{+}$.
Let $d_{-}$ be the point in $\partial D_{-}$ that is identified with $d_{+}$.
Then $d_{-}$ is not in $\Delta_{-}$.
Take a point $q$ in $\alpha_{-}$ between $p$ and $c_{-}$.
Let $\delta_{+}$ be an arc in $\overline{\Delta_{-} - (D_{+} \cup \Delta_{+})}$ connecting $d_{+}$ and $q$.
Since $D_{-} \cup \Delta_{-}$ is just a disk, we can take a non-separating arc $\delta_{-}$ in $\overline{\Sigma' - (D_{-} \cup \Delta_{-})}$ connecting $q$ and $d_{-}$ (see Figure \ref{fig4}).

\begin{figure}[!hbt]
\includegraphics[width=7cm,clip]{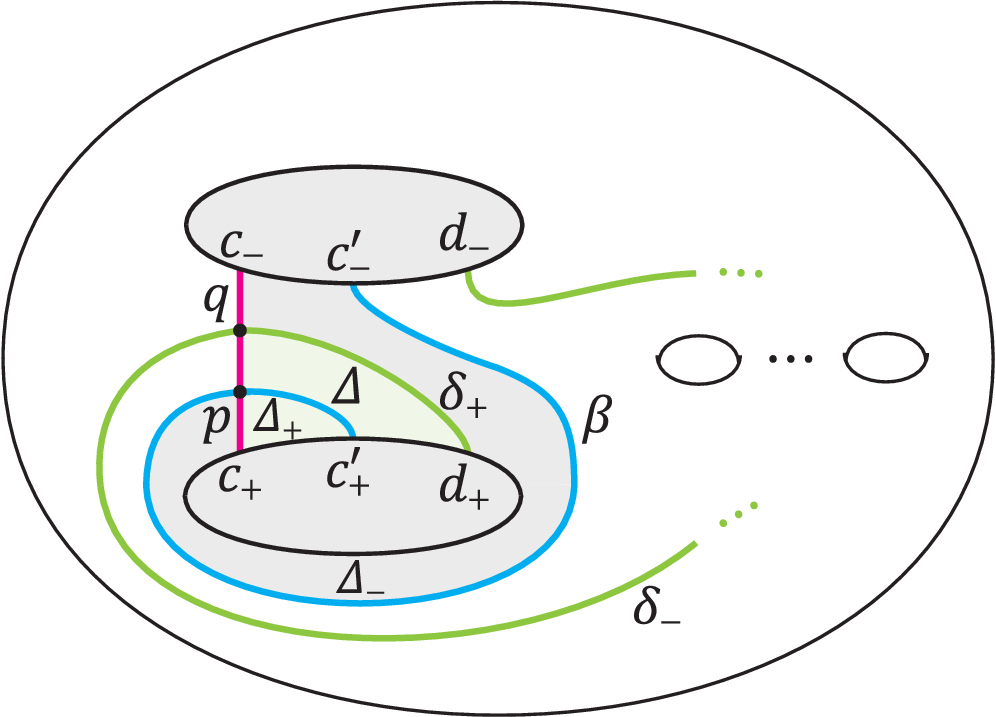}
\caption{The arc $\delta_{+}$ and $\delta_{-}$.}
\label{fig4}
\end{figure}

Let $\delta = \delta_{+} \cup \delta_{-}$.
If we identify $D_{+}$ and $D_{-}$, then $\delta$ becomes a primitive curve with a dual disk $D$.
The arc $\delta$ is disjoint from $\beta$, and $\delta$ intersects $\alpha$ in a single point $q$.
Now the situation for $\alpha$ and $\delta$ is the same as the case (a) above, so $\delta$ is c-connected to $\alpha$.

\vspace{0.3cm}

\noindent Case $2$. At least one of $\alpha_{+} \cup \beta_{+}$ and $\alpha_{-} \cup \beta_{-}$, say $\alpha_{+} \cup \beta_{+}$, does not cut off a disk from $\Sigma' - D_{+}$.

Without loss of generality, assume that $\beta_{+}$ is incident to $p$ in the right side of $\alpha$, and then take a point $d_{+}$ in $\partial D_{+}$ near $c_{+}$ and in the right side of $c_{+}$ as in Figure \ref{fig5}.
Take a point $q$ in $\alpha_{+}$ near $c_{+}$.
Let $\delta_{+}$ be a short arc connecting $d_{+}$ and $q$ such that the interior of $\delta_{+}$ is disjoint from $D_{+} \cup D_{-} \cup \alpha \cup \beta$.
Let $\delta_{-} =$ (the subarc of $\alpha_{+}$ from $q$ to $p) \cup \beta_{-}$.
Slide the endpoint of $\delta_{-}$ in $\partial D_{-}$ to the point $d_{-}$ that is identified with $d_{+}$ in $\Sigma$.
Then isotope $\delta_{-}$ slightly so that $\delta = \delta_{+} \cup \delta_{-}$ intersects $\alpha$ in a single point $q$, and $\delta$ is disjoint from $\beta$ (see Figure \ref{fig5}).

\begin{figure}[!hbt]
\includegraphics[width=7cm,clip]{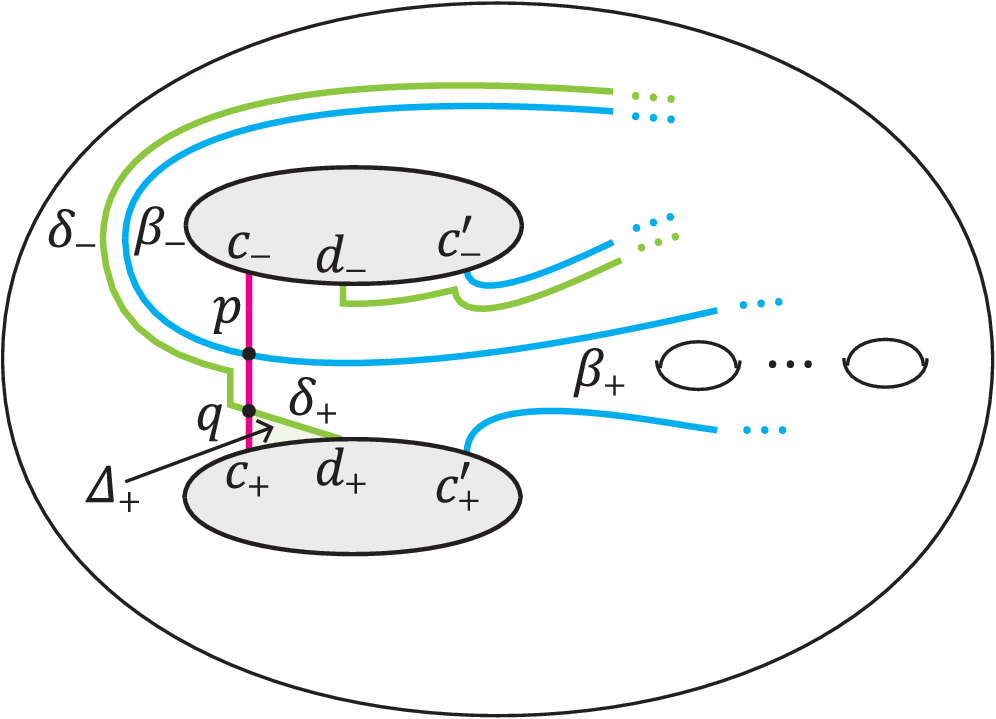}
\caption{The arc $\delta = \delta_{+} \cup \delta_{-}$.}
\label{fig5}
\end{figure}

Let $\alpha'_{+}$ be the subarc of $\alpha_{+}$ from $c_{+}$ to $q$, and $\alpha'_{-}$ be the subarc of $\alpha$ from $q$ to $c_{-}$.
The union $\alpha'_{+} \cup \delta_{+}$ cuts off a small disk $\Delta_{+}$ from $\overline{\Sigma' - D_{+}}$.
If $\alpha'_{-} \cup \delta_{-}$ also cuts off a disk from $\overline{\Sigma' - D_{-}}$, then we apply Case $1$ to $\alpha$ and $\delta$.
If $\alpha'_{-} \cup \delta_{-}$ does not cut off a disk from $\overline{\Sigma' - D_{-}}$, then we apply Case $2$ again to $\alpha$ and $\delta$ with respect to $\alpha'_{-} \cup \delta_{-}$.
\end{proof}

\begin{lemma}[A stronger version of Lemma \ref{lem2.1}]\label{lem2.2}
Let $C$ and $C'$ be primitive curves in $\Sigma$, the boundary of a genus-$g$ handlebody $V$ for $g \ge 2$.
Suppose that $C$ and $C'$ have a common dual disk $D$.
Then $C$ and $C'$ are c-connected with the common dual disk $D$.
\end{lemma}

\begin{proof}
If $C \cap C' = \emptyset$, then we are done.
If $| C \cap C' | = 1$, then the result holds by Lemma \ref{lem2.1}.
So suppose that $| C \cap C' | \ge 2$.
Let $p = C \cap \partial D$ and $p' = C' \cap \partial D$.

Consider a subarc $\alpha$ of $C'$ cut off by $C$ such that the interior of $\alpha$ is disjoint from $C$.
Let $a$ and $b$ be the two endpoints of $\alpha$.
The two points $a$ and $b$ cut $C$ into two arcs $\gamma_1$ and $\gamma_2$.
If $\alpha$ contains $p'$, then let $\gamma_1$ be the arc that does not contain $p$.
If $\alpha$ does not contain $p'$, then let $\gamma_1$ be the arc that contains $p$.
So in any case, $\gamma_1 \cup \alpha$ intersects $\partial D$ in a single point.
There are two cases according to the sides that $\alpha$ is incident to $C$ at $a$ and $b$.

\vspace{0.3cm}

\noindent Case $1$. $\alpha$ is incident to $C$ in the same side of $C$ at $a$ and $b$.

By a surgery of $C$ along $\alpha$, we obtain a new primitive curve $C_2 = \gamma_1 \cup \alpha$ with a dual disk $D$.
After a slight isotopy, $C_2$ is disjoint from $C_1(= C)$, and $| C_2 \cap C' | \le | C_1 \cap C' | - 2$ (see Figure \ref{fig6}).

\begin{figure}[!hbt]
\includegraphics[width=7.5cm,clip]{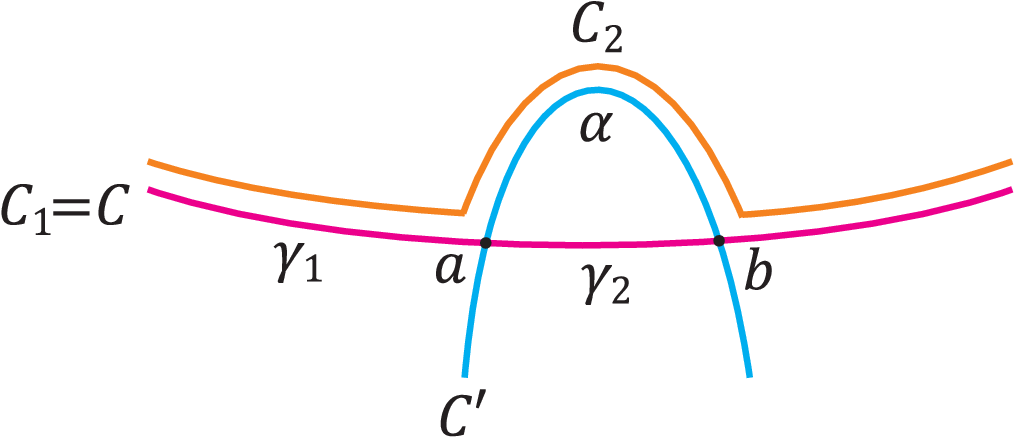}
\caption{The primitive curve $C_2$ in Case $1$.}
\label{fig6}
\end{figure}

\vspace{0.3cm}

\noindent Case $2$. $\alpha$ is incident to $C$ in the opposite sides of $C$ at $a$ and $b$.

By a surgery of $C$ along $\alpha$, we obtain a new primitive curve $\Gamma = \gamma_1 \cup \alpha $ with a dual disk $D$.
After a slight isotopy, $| \Gamma \cap C | = 1$, and $| \Gamma \cap C' | \le | C \cap C' | - 1$ (see Figure \ref{fig7}).
By Lemma \ref{lem2.1}, $C$ and $\Gamma$ are c-connected via primitive curves with a common dual disk $D$.

\begin{figure}[!hbt]
\includegraphics[width=7cm,clip]{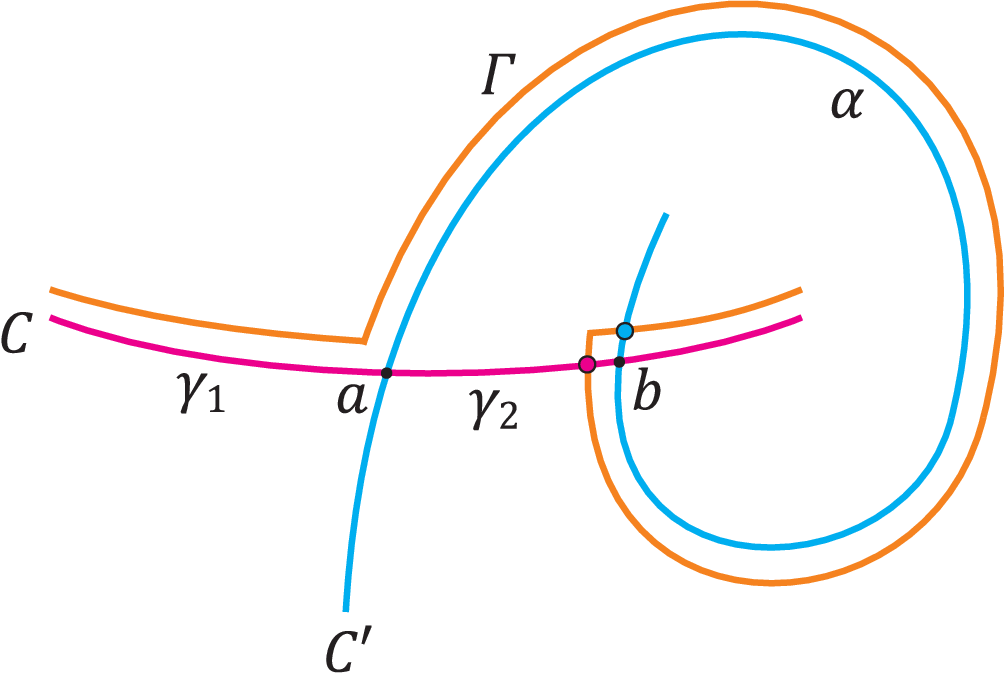}
\caption{The primitive curve $\Gamma$ in Case $2$.}
\label{fig7}
\end{figure}

By an inductive argument, we get a sequence of primitive curves with a common dual disk $D$, from $C$ to $C'$.
\end{proof}

A key idea in the next lemma is that we do arc surgery with respect to a pair of a primitive curve and a boundary curve of a dual disk.

\begin{lemma}[Arc surgery]\label{lem2.3}
Let $C$ and $C'$ be primitive curves in $\Sigma$, the boundary of a genus-$g$ handlebody $V$ for $g \ge 2$.
Let $D$ and $D'$ be dual disks of $C$ and $C'$ respectively.
Suppose that $D$ and $D'$ are disjoint.
Then there exists a primitive curve $C''$ in $\Sigma$ such that $| C'' \cap \partial D' | \le 1$, and $C$ and $C''$ are c-connected with the common dual disk $D$.
\end{lemma}

\begin{proof}
If $| C \cap \partial D' | \le 1$, then we are done.
So suppose that $| C \cap \partial D' | \ge 2$.
Let $p = C \cap \partial D$ and $p' = C' \cap \partial D'$.

Consider a subarc $\alpha$ of $\partial D'$ cut off by $C$ such that the interior of $\alpha$ is disjoint from $C$.
In particular, we choose such an arc component $\alpha$ that does not contain $p'$.
(A choice like this will be used again in the proof of Theorem \ref{thm2.4}.)
Let $a$ and $b$ be the two endpoints of $\alpha$.
The two points $a$ and $b$ cut $C$ into two arcs $\gamma_1$ and $\gamma_2$, and let $\gamma_1$ be the arc that contains $p$.
There are two cases according to the sides that $\alpha$ is incident to $C$ at $a$ and $b$.

\vspace{0.3cm}

\noindent Case $1$. $\alpha$ is incident to $C$ in the same side of $C$ at $a$ and $b$.

By a surgery of $C$ along $\alpha$, we obtain a new primitive curve $C_2 = \gamma_1 \cup \alpha$ with a dual disk $D$.
After a slight isotopy, $C_2$ is disjoint from $C_1(= C)$, and $| C_2 \cap \partial D' | \le | C_1 \cap \partial D' | - 2$ (see Figure \ref{fig8}).

\begin{figure}[!hbt]
\includegraphics[width=8cm,clip]{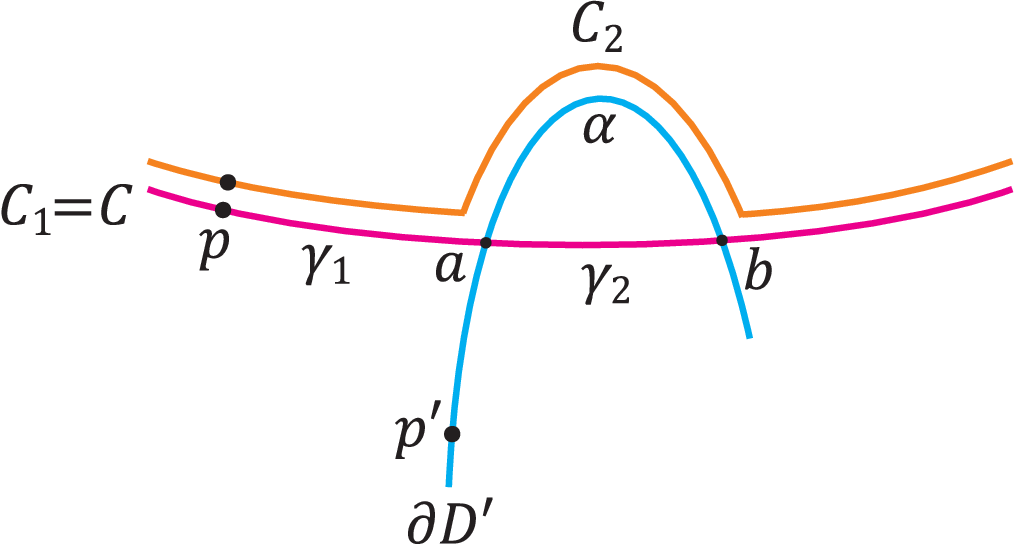}
\caption{The primitive curve $C_2$ in Case $1$.}
\label{fig8}
\end{figure}

\vspace{0.3cm}

\noindent Case $2$. $\alpha$ is incident to $C$ in the opposite sides of $C$ at $a$ and $b$.

By a surgery of $C$ along $\alpha$, we obtain a new primitive curve $\Gamma = \gamma_1 \cup \alpha $ with a dual disk $D$.
After a slight isotopy, $| \Gamma \cap C | = 1$, and $| \Gamma \cap \partial D' | \le | C \cap \partial D' | - 1$ (see Figure \ref{fig9}).
By Lemma \ref{lem2.1}, $C$ and $\Gamma$ are c-connected via primitive curves with a common dual disk $D$.

\begin{figure}[!hbt]
\includegraphics[width=7.5cm,clip]{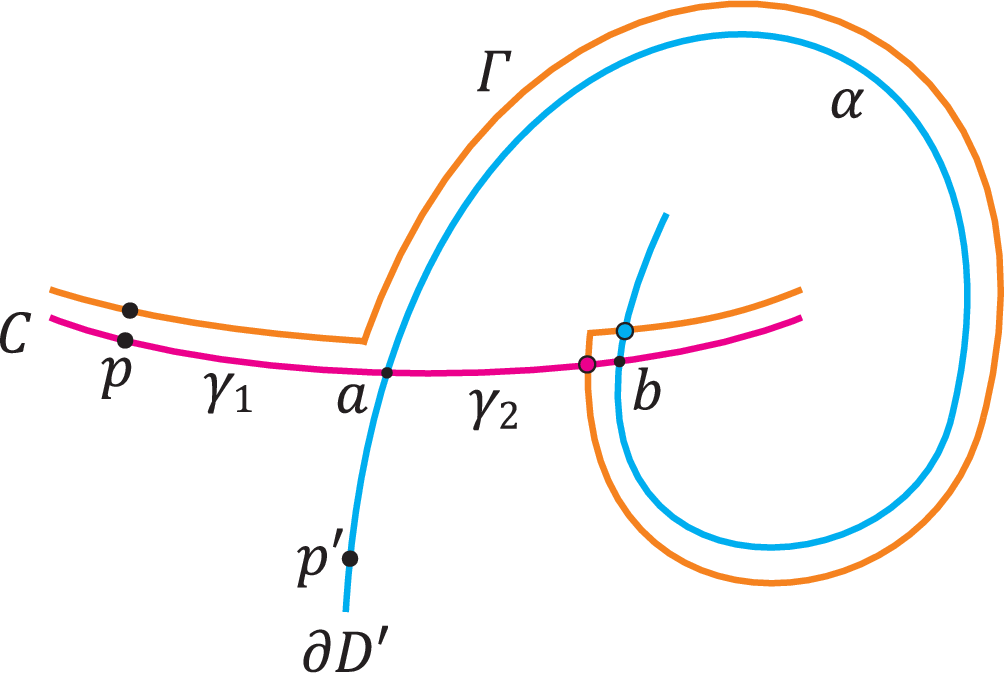}
\caption{The primitive curve $\Gamma$ in Case $2$.}
\label{fig9}
\end{figure}

Note that $| C_2 \cap C' | \le | C_1 \cap C' |$ in Case $1$ and $| \Gamma \cap C' | \le | C \cap C' |$ in Case $2$ since $\alpha$ does not contain $p'$.
By an inductive argument, we get a primitive curve $C''$ with a dual disk $D$ and $| C'' \cap \partial D' | \le 1$.
\end{proof}

Now we are ready to prove the following theorem, a weaker version of Theorem \ref{thm1.1}.

\begin{theorem}\label{thm2.4}
Let $C$ and $C'$ be primitive curves in $\Sigma$, the boundary of a genus-$g$ handlebody $V$ for $g \ge 2$.
Then $C$ and $C'$ are c-connected.
\end{theorem}

\begin{proof}
Take any dual disks $D$ and $D'$ of $C$ and $C'$ respectively.
By the standard disk surgery argument, there exists a sequence $D = D_1, D_2, \ldots, D_n$ $= D'$ of non-separating disks in $V$ such that $D_j$ and $D_{j+1}$ are disjoint for each $j \in \{ 1, 2, \ldots, n-1 \}$.
Let $C_1 = C$ and $C_n = C'$.
Take a primitive curve $C_j$ for each $j \in \{ 2, \ldots, n-1 \}$ such that $D_j$ is a dual disk of $C_j$.

To show the statement of Theorem \ref{thm2.4} for $C$ and $C'$, it is enough to show the same statement for $C_j$ and $C_{j+1}$, and each $j \in \{ 1, 2, \ldots, n-1 \}$.
So without loss of generality, we assume that $C$ and $C'$ admit dual disks $D$ and $D'$ respectively such that $D$ and $D'$ are disjoint.

By Lemma \ref{lem2.3}, there exists a primitive curve $C''$ such that $C$ and $C''$ are c-connected with a common dual disk $D$ and $| C'' \cap \partial D' | \le 1$.
Similarly, there exists a primitive curve $C'''$ such that $C'$ and $C'''$ are c-connected with a common dual disk $D'$ and $| C''' \cap \partial D | \le 1$.
It remains to show that $C''$ and $C'''$ are c-connected.

Suppose that one of $| C'' \cap \partial D' |$ or $| C''' \cap \partial D |$, say $| C'' \cap \partial D' |$, is $1$.
Then $D'$ is a common dual disk of $C''$ and $C'''$.
By Lemma \ref{lem2.2}, $C''$ and $C'''$ are c-connected.

Now suppose that both $| C'' \cap \partial D' |$ and $| C''' \cap \partial D |$ are $0$.
We do an arc surgery for the pair $C''$ and $C'''$ to reduce $| C'' \cap C''' |$.
More precisely, we do a surgery of $C''$ along an arc component $\alpha$ of $C'''$ cut off by $C''$ such that the interior of $\alpha$ is disjoint from $C''$ and that $\alpha$ does not contain the point $C''' \cap \partial D'$ (as in the proof of Lemma \ref{lem2.3}).
Then we obtain a primitive curve $\Gamma_1$ satisfying $| \Gamma_1 \cap C''' | < | C'' \cap C''' |$ with a common dual disk $D$, and with an additional property that $| \Gamma_1 \cap \partial D' | = 0$ since $| \Gamma_1 \cap \partial D' | \le | C'' \cap \partial D' | = 0$ by the choice of $\alpha$.
We iterate such an arc surgery to obtain a primitive curve $\Gamma$ from $C''$ with a common dual disk $D$, and $| \Gamma \cap C''' | \le 1$, and $| \Gamma \cap \partial D' | = 0$.
The two primitive curves $C''$ and $\Gamma$ are c-connected (using Lemma \ref{lem2.1} if necessary).
See Figure \ref{fig10}.

\begin{figure}[!hbt]
\includegraphics[width=5.5cm,clip]{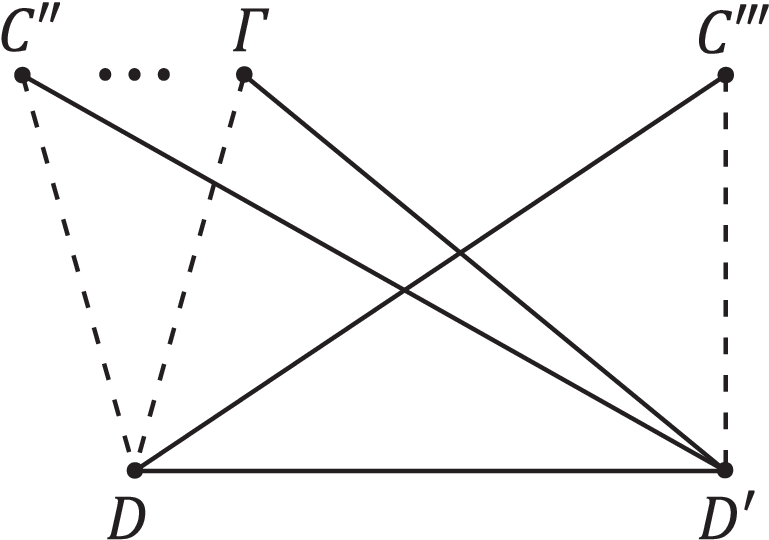}
\caption{$| (\Gamma \cup D) \cap (C''' \cup D') | \le 1$.}
\label{fig10}
\end{figure}

Since $| (\Gamma \cup D) \cap (C''' \cup D') | = | \Gamma \cap C''' | \le 1$, we can take a band sum of $D$ and $D'$ that is a common dual disk of $\Gamma$ and $C'''$.
See Figure \ref{fig11}.

\begin{figure}[!hbt]
\includegraphics[width=12cm,clip]{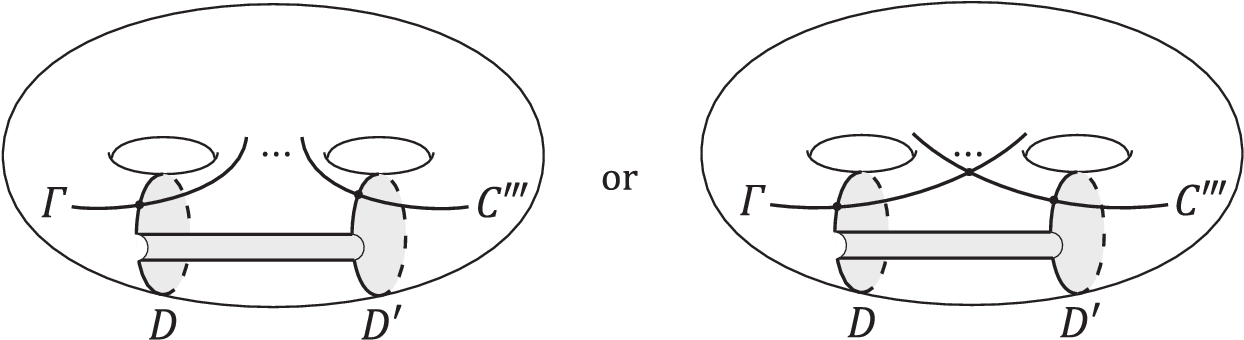}
\caption{A band sum of $D$ and $D'$ is a common dual disk of $\Gamma$ and $C'''$.}
\label{fig11}
\end{figure}
If $| \Gamma \cap C''' | = 0$, then $\Gamma$ and $C'''$ are c-connected obviously.
If $| \Gamma \cap C''' | = 1$, then $\Gamma$ and $C'''$ are again c-connected by Lemma \ref{lem2.1}.
\end{proof}

\section{Proof of the main theorem in the case of $g = 2$}\label{sec3}

For convenience, we simply say that two primitive curves $C$ and $C'$ {\em are s-connected} if there exists a sequence $C = C_1, C_2, \ldots, C_k = C'$ of primitive curves for some $k$ such that $C_i$ and $C_{i+1}$ are separated for each $i \in \{ 1, 2, \ldots, k-1 \}$.
If $C$ and $C'$ are s-connected, that is, $C_i$ and $C_{i+1}$ are separated with their disjoint dual disks $D_i$ and $D_{i+1}$ respectively, then we see that $C$ and $C'$ are also c-connected by taking a band sum of $D_i$ and $D_{i+1}$ that forms a common dual disk of $C_i$ and $C_{i+1}$.
We will prove the following, Theorem \ref{thm1.1} in the case of $g = 2$.

\begin{theorem}\label{thm3.1}
Let $C$ and $C'$ be primitive curves in $\Sigma$, the boundary of a genus-$2$ handlebody $V$.
Then $C$ and $C'$ are s-connected.
\end{theorem}

\begin{proof}
By Theorem \ref{thm2.4}, $C$ and $C'$ are c-connected (with an arbitrary choice of dual disks of $C$ and $C'$), and hence without loss of generality we may assume that $C$ and $C'$ are disjoint and have a common dual disk.

Let $D$ be a common dual disk of $C$ and $C'$.
Cutting $V$ along $D$, we have a solid torus $V'$ with two copies $D_{+}$ and $D_{-}$ of $D$ on $\Sigma' = \partial V'$.
The primitive curve $C$ is cut into an arc $\alpha$ with one of its endpoints in $\partial D_{+}$ and the other in $\partial D_{-}$.
The primitive curve $C'$ is also cut into an arc $\beta$ with its endpoints in $\partial D_{+}$ and $\partial D_{-}$ respectively.
The union $\alpha \cup \beta \cup D_{+} \cup D_{-}$ is homotopy equivalent to an inessential loop in $\Sigma'$ only if $C$ and $C'$ are isotopic.
So we may assume that $\alpha \cup \beta \cup D_{+} \cup D_{-}$ is homotopy equivalent to an essential loop in the torus $\Sigma' = \partial V'$.

\vspace{0.2cm}

Suppose that $\alpha \cup \beta \cup D_{+} \cup D_{-}$ is homotopy equivalent to a meridian.
See Figure \ref{fig12}.
Figure \ref{fig13} illustrates that $C$ and $C'$ are s-connected via $C''$ and $C'''$.

\begin{figure}[!hbt]
\includegraphics[width=13cm,clip]{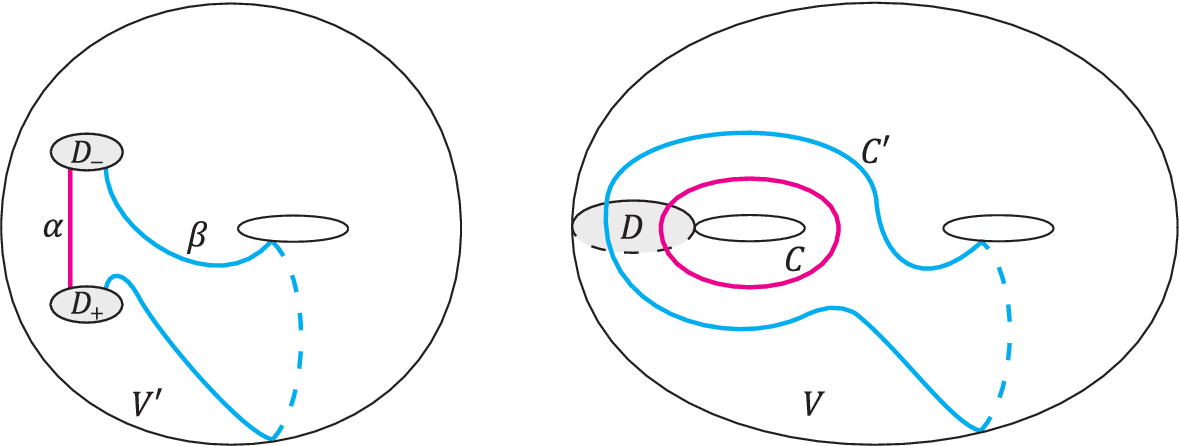}
\caption{$\alpha \cup \beta \cup D_{+} \cup D_{-}$ is homotopy equivalent to a meridian.}
\label{fig12}
\end{figure}

\begin{figure}[!hbt]
\includegraphics[width=15cm,clip]{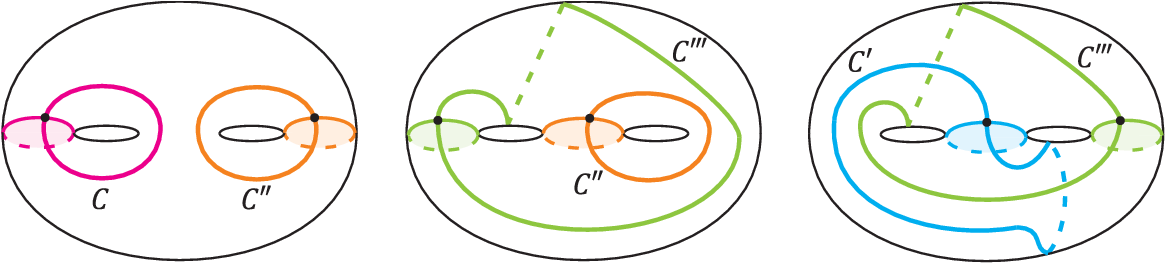}
\caption{$C$ and $C'$ are $s$-connected via $C''$ and $C'''$.}
\label{fig13}
\end{figure}

Suppose that $\alpha \cup \beta \cup D_{+} \cup D_{-}$ is homotopy equivalent to a longitude.
Figure \ref{fig14} illustrates that $C$ and $C'$ are separated.
So they are s-connected directly.

\begin{figure}[!hbt]
\includegraphics[width=13cm,clip]{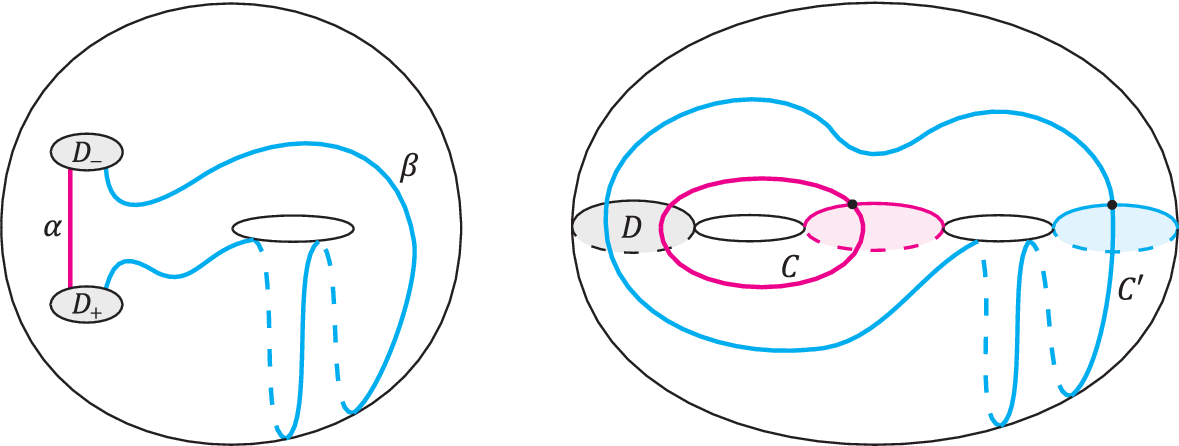}
\caption{$C$ and $C'$ are separated.}
\label{fig14}
\end{figure}

Suppose that $\alpha \cup \beta \cup D_{+} \cup D_{-}$ is homotopy equivalent to a $(p, q)$-torus knot $K_{p, q}$ with $p \ge 2$ and $p > q$ without loss of generality.
Here $K_{p, q}$ winds $V'$ $p$ times in longitudinal direction and $q$ times in meridional direction.

We use induction on the absolute values of $p$ and $q$.
For that purpose, we consider two types of arcs $\beta_{p, q}$ and $\beta'_{p, q}$ as in Figure \ref{fig15} (a) and Figure \ref{fig15} (b) respectively.
Both $\beta_{p, q}$ and $\beta'_{p, q}$ have one of their endpoints in $\partial D_{+}$ and the other in $\partial D_{-}$.
The arc $\beta_{p, q}$ is disjoint from $\alpha$, but $\beta'_{p, q}$ intersects $\alpha$ in a single point, cutting off a small triangular disk together with $D_{-}$.
Both $\beta_{p, q} \cup \alpha \cup D_{+} \cup D_{-}$ and $\beta'_{p, q} \cup \alpha \cup D_{+} \cup D_{-}$ are homotopy equivalent to $K_{p, q}$.
For an example, $(p, q) = (7, 5)$ in Figure \ref{fig15}.

\begin{figure}[!hbt]
\includegraphics[width=15cm,clip]{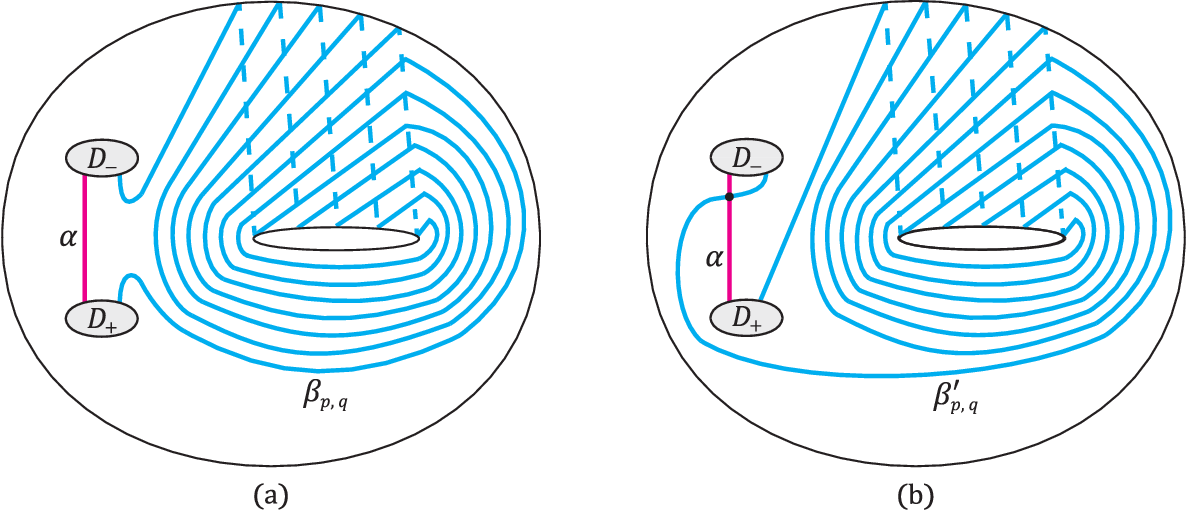}
\caption{(a) $| \beta_{p, q} \cap \alpha | = 0$ and (b) $| \beta'_{p, q} \cap \alpha | = 1$.}
\label{fig15}
\end{figure}

Let $r$ and $s$ be positive integers satisfying $ps - qr = 1$ and $r < p$ and $s \le q$.
Since $\det
\left(
\begin{smallmatrix}
p & r \\
q & s
\end{smallmatrix}
\right)
= ps - qr = 1$, $| K_{p, q} \cap K_{r, s} | = 1$.
For $\beta_{p, q}$ and $\beta'_{p, q}$, we take arcs $\beta_{r, s}$ and $\beta'_{r, s}$ respectively winding $r$ times in longitudinal direction and $s$ times in meridional direction and $| \beta_{r, s} \cap \beta_{p, q} | = | \beta'_{r, s} \cap \beta'_{p, q} | = 1$ as in Figure \ref{fig16} (a) and Figure \ref{fig16} (b).
Both $\beta_{r, s}$ and $\beta'_{r, s}$ have one of their endpoints in $\partial D_{+}$ and the other in $\partial D_{-}$.
Moreover, $| \beta_{r, s} \cap \alpha | = 0$ and $| \beta'_{r, s} \cap \alpha | = 1$.
The arc $\beta_{r, s}$ is like a stepping stone between $\alpha$ and $\beta_{p, q}$ and $\beta'_{r, s}$ is like a stepping stone between $\alpha$ and $\beta'_{p, q}$.

\begin{figure}[!hbt]
\includegraphics[width=15cm,clip]{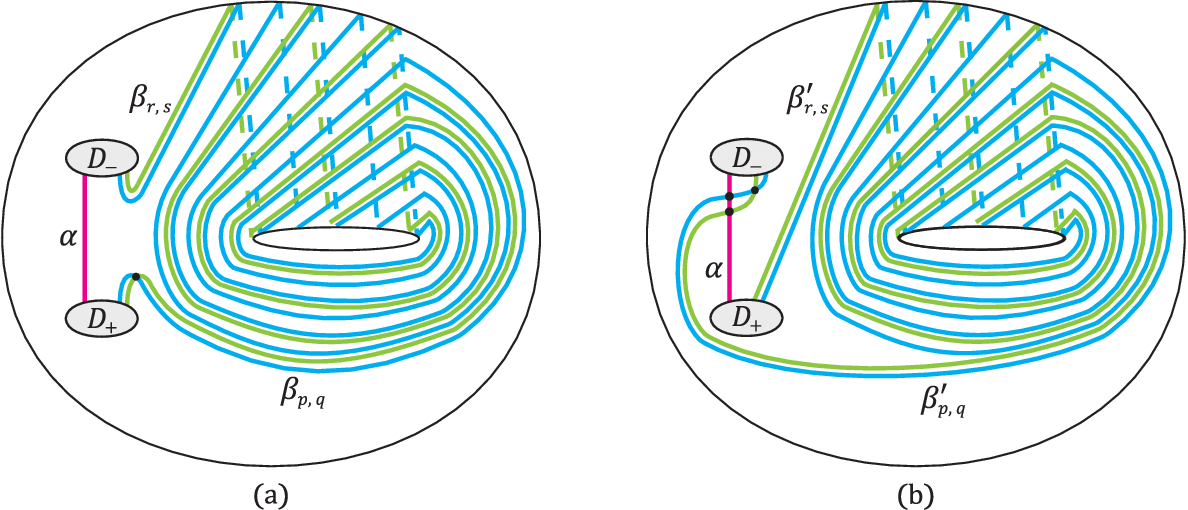}
\caption{(a) $| \beta_{r, s} \cap \alpha | = 0$ and (b) $| \beta'_{r, s} \cap \alpha | = 1$, and $| \beta_{r, s} \cap \beta_{p ,q} | = | \beta'_{r,s} \cap \beta'_{p ,q} | = 1$.}
\label{fig16}
\end{figure}

The disk $D_{-}$ in Figure \ref{fig16} (a) is isotoped along two long parallel subarcs of $\beta_{r, s}$ and $\beta_{p, q}$ as in Figure \ref{fig17} (a) so that $\beta_{r, s}$ becomes a short arc and $\beta_{p, q}$ winds $p - r$ times in longitudinal direction.
Denote the isotoped $\beta_{r, s}$ and $\beta_{p, q}$ by $\alpha_0$ and $\beta'_{p - r, t}$ for some $t < q$ respectively.
Similarly, $D_{+}$ in Figure \ref{fig16} (b) is isotoped along two long parallel subarcs of  $\beta'_{r, s}$ and $\beta'_{p, q}$ as in Figure \ref{fig17} (b) so that $\beta'_{r, s}$ becomes a short arc and $\beta'_{p, q}$ winds $p - r$ times in longitudinal direction.
Denote the isotoped $\beta'_{r, s}$ and $\beta'_{p, q}$ by $\alpha_0$ and $\beta'_{p - r, t}$ for some $t < q$ respectively.
In both cases, $\alpha_0 \cup \beta'_{p - r, t} \cup D_{+} \cup D_{-}$ is homotopy equivalent to a $K_{p - r, t}$.
Since $K_{p - r, t}$ intersects the original $K_{p, q}$ in a single point and
$\det
\left(
\begin{smallmatrix}
p & p - r \\
q & q - s
\end{smallmatrix}
\right)
= p(q - s) - q(p - r) = -1$, we have that $t = q - s$.
Hence the longitudinal and meridional parameters of $\beta_{p, q}$ with respect to $\beta_{r, s}$ and those of $\beta'_{p, q}$ with respect to $\beta'_{r, s}$ are $(p - r, q - s)$.
In the example of Figure \ref{fig15}, Figure \ref{fig16}, Figure \ref{fig17}, $(p, q) = (7, 5)$ and $(r, s) = (4, 3)$ and $(p - r, q - s) = (3, 2)$.

\begin{figure}[!hbt]
\includegraphics[width=15cm,clip]{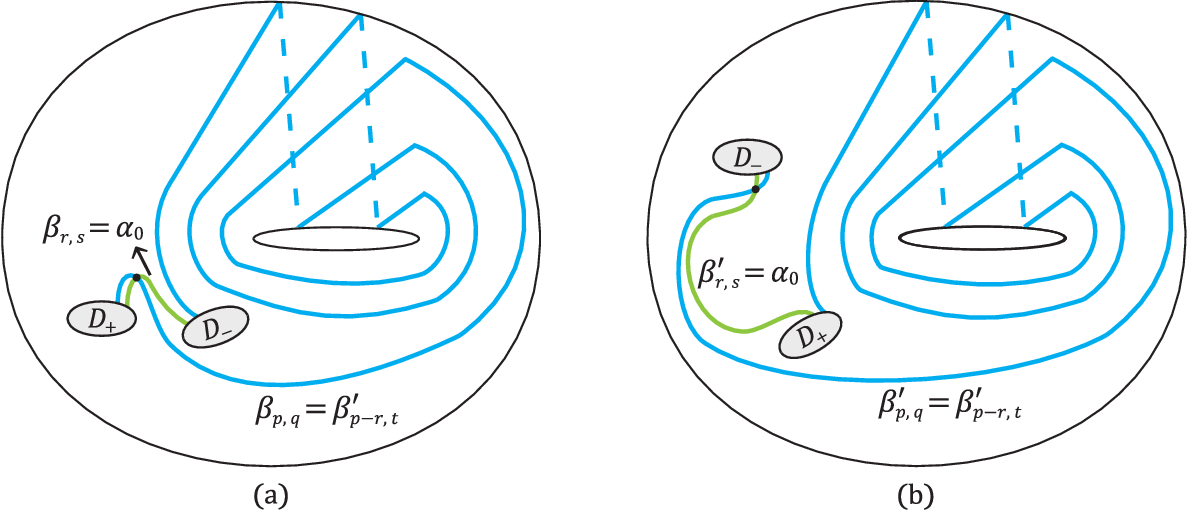}
\caption{$\beta'_{p - r, t} = \beta'_{p - r, q - s}$.}
\label{fig17}
\end{figure}

Now the initial conditions of the induction, $(p, q) = (1, 0)$ and $(1, 1)$ are remained.
We only need to consider $| \beta'_{1,0} \cap \alpha | = | \beta'_{1,1} \cap \alpha | = 1$ because we already dealt the cases of $| \beta_{1,0} \cap \alpha | = | \beta_{1,1} \cap \alpha | = 0$ (Figure \ref{fig14}).
Figure \ref{fig18} shows that $C$ and $C'$ are s-connected via $C''$, where $C$ and $C'$ are obtained from $\alpha$ and $\beta'_{1,0}$ respectively.
(It is easy to see that $C$ and $C''$ are separated.)
Figure \ref{fig19} shows that $C$ and $C'$ are s-connected via $C''$, where $C$ and $C'$ are obtained from $\alpha$ and $\beta'_{1,1}$ respectively.
(It is easy to see that $C$ and $C''$ are separated.)

\begin{figure}[!hbt]
\includegraphics[width=15cm,clip]{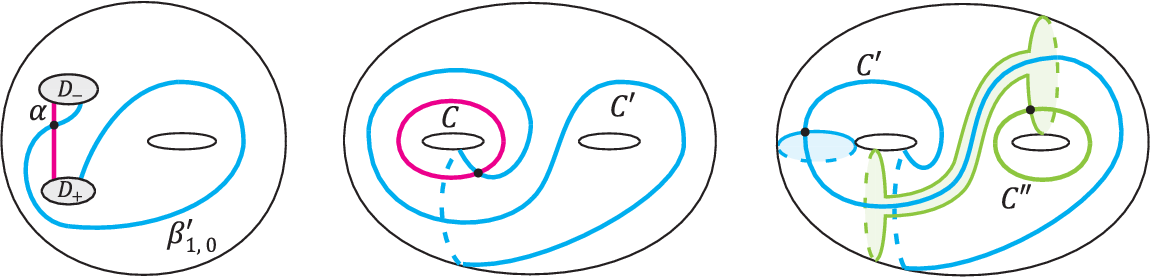}
\caption{$C$ and $C'$ are s-connected via $C''$.}
\label{fig18}
\end{figure}

\begin{figure}[!hbt]
\includegraphics[width=15cm,clip]{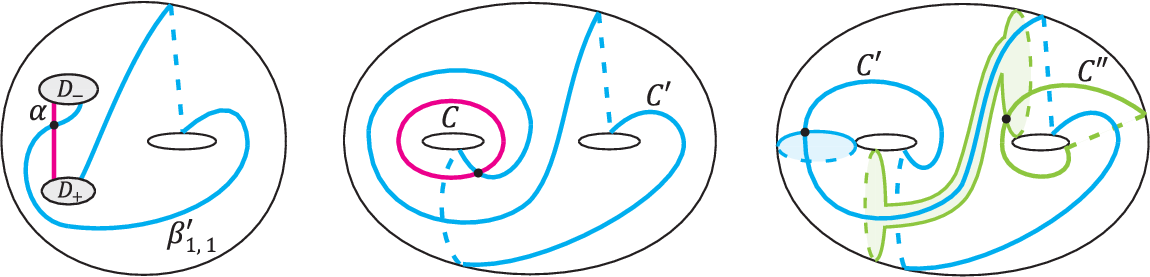}
\caption{$C$ and $C'$ are s-connected via $C''$.}
\label{fig19}
\end{figure}

Let $C$, $C''$, $C'$ be the primitive curves in $\partial V$ obtained by identifying endpoints of $\alpha$, $\beta_{r, s}$(or $\beta'_{r, s}$), $\beta_{p, q}$(or $\beta'_{p, q}$) respectively.
Since $r < p$ and $s \le q$, $C$ and $C''$ are s-connected by our induction hypothesis.
Since $p - r < p$ and $q - s < q$, $C''$ and $C'$ are s-connected by our induction hypothesis.
Hence $C$ and $C'$ are s-connected.
\end{proof}

\section{Proof of the main theorem in the case of $g \ge 3$}\label{sec4}

Throughout the section, $V$ will be assumed to be a genus-$g$ handlebody for $g \ge 3$.
For a primitive curve $C$ in $\Sigma = \partial V$ with a dual disk $D$, we call the pair $(C, D)$ simply a {\em dual pair} for $V$.
For convenience, we simply say that two dual pairs $(C, D)$ and $(C', D')$ {\em are p-connected} if there exists a sequence $(C, D) = (C_1, D_1), (C_2, D_2), \ldots, (C_k, D_k) = (C', D')$ of dual pairs for some $k$ such that $C_i \cup D_i$ and $C_{i+1} \cup D_{i+1}$ are disjoint for each $i \in \{ 1, 2, \ldots, k-1 \}$.
It is obvious that if two dual pairs $(C, D)$ and $(C', D')$ are p-connected then $C$ and $C'$ are s-connected.
The following Lemma \ref{lem4.1}, Lemma \ref{lem4.2}, and Lemma \ref{lem4.3} hold only when $g \ge 3$.

\begin{lemma}[Common primitive curve]\label{lem4.1}
Let $(C, D)$ and $(C, D')$ be dual pairs for $V$.
Suppose that $D$ and $D'$ are disjoint.
Then  $(C, D)$ and $(C, D')$ are p-connected.
\end{lemma}

\begin{proof}
Cut $V$ along $D \cup D'$.
The resulting manifold $V'$ is a genus-$(g-2)$ handlebody, or two handlebodies of genus $g_1$ and $g_2$ respectively, where $g_1 + g_2 = g-1$.
The primitive curve $C$ is cut into two arcs $\alpha_1$ and $\alpha_2$.
Whether $V'$ is connected or not, there are two copies $D_{+}$ and $D_{-}$ of $D$ and two copies $D'_{+}$ and $D'_{-}$ of $D'$ on $\Sigma' = \partial V'$.
One of $\alpha_1$ and $\alpha_2$, say $\alpha_1$, connects $D_{+}$ and $D'_{-}$ and the other $\alpha_2$ connects $D'_{+}$ and $D_{-}$.
Since each of $D_{+} \cup \alpha_1 \cup D'_{-}$ and $D'_{+} \cup \alpha_2 \cup D_{-}$ is homotopy equivalent to a point, we can take a dual pair $(C'', D'')$ in $V'$ disjoint from them.
See Figure \ref{fig20}.
Then $(C'', D'')$ can be regarded as a dual pair also in $V$, and $(C, D)$ and $(C, D')$ are p-connected via $(C'', D'')$.
\end{proof}

\begin{figure}[!hbt]
\includegraphics[width=13cm,clip]{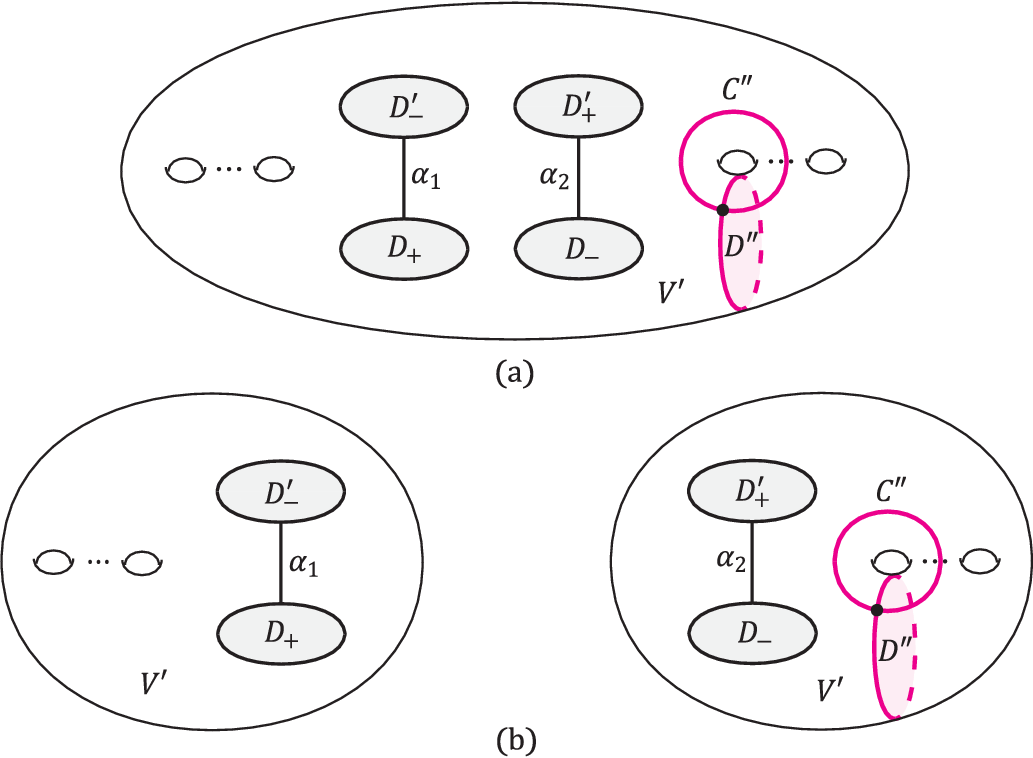}
\caption{(a) $V'$ is a handlebody and (b) $V'$ is two handlebodies.}
\label{fig20}
\end{figure}

\begin{lemma}[Common dual disk]\label{lem4.2}
Let $(C, D)$ and $(C', D)$ be dual pairs for $V$.
Suppose that $C$ and $C'$ are disjoint.
Then  $(C, D)$ and $(C', D)$ are p-connected.
\end{lemma}

\begin{proof}
Cut $V$ along $D$.
Then we have a genus-$(g-1)$ handlebody $V'$ with two copies $D_{+}$ and $D_{-}$ of $D$ on $\Sigma' = \partial V'$.
The primitive curves $C$ and $C'$ are cut into arcs $\alpha$ and $\alpha'$ in $\Sigma'$ respectively such that for each arc one endpoint is in $\partial D_{+}$ and the other endpoint is in $\partial D_{-}$.
Since $D_{+} \cup \alpha' \cup D_{-}$ is homotopy equivalent to a point, we can take a dual pair $(C'', D'')$ in $V'$ disjoint from it.
Then $(C'', D'')$ can be regarded as a dual pair also in $V$, and $(C', D)$ and $(C'', D'')$ are p-connected.
If $(C'' \cup D'') \cap \alpha = \emptyset$, then $(C'', D'')$ and $(C, D)$ are also p-connected and we are done.
So we may assume that $(C'' \cup D'') \cap \alpha \ne \emptyset$.

We use induction on $| (C'' \cup D'') \cap \alpha |$.
Let $p$ be a point of $(C'' \cup D'') \cap \alpha$ which is closest to one of the two endpoints of $\alpha$, say $c_{+}$, in $\partial D_{+}$.
There are two cases.

\vspace{0.3cm}

\noindent Case $1$. $p \in C''$.

Let $C''_1 = C''$, and $C''_2$ be a parallel copy of $C''_1$ such that a point $q \in C''_2 \cap \alpha$ is closer to $c_{+}$ than $p$.
Let $\alpha_{+}$ be the subarc of $\alpha$ from $q$ to $c_{+}$.
Slide a small neighborhood of $q$ in $C''_2$ along $\alpha_{+} \cup D_{+}$.
Then $C''_1 \cup C''_2$ bounds an annulus $A$ containing $D_{+}$.
See Figure \ref{fig21}.

\begin{figure}[!hbt]
\includegraphics[width=10.5cm,clip]{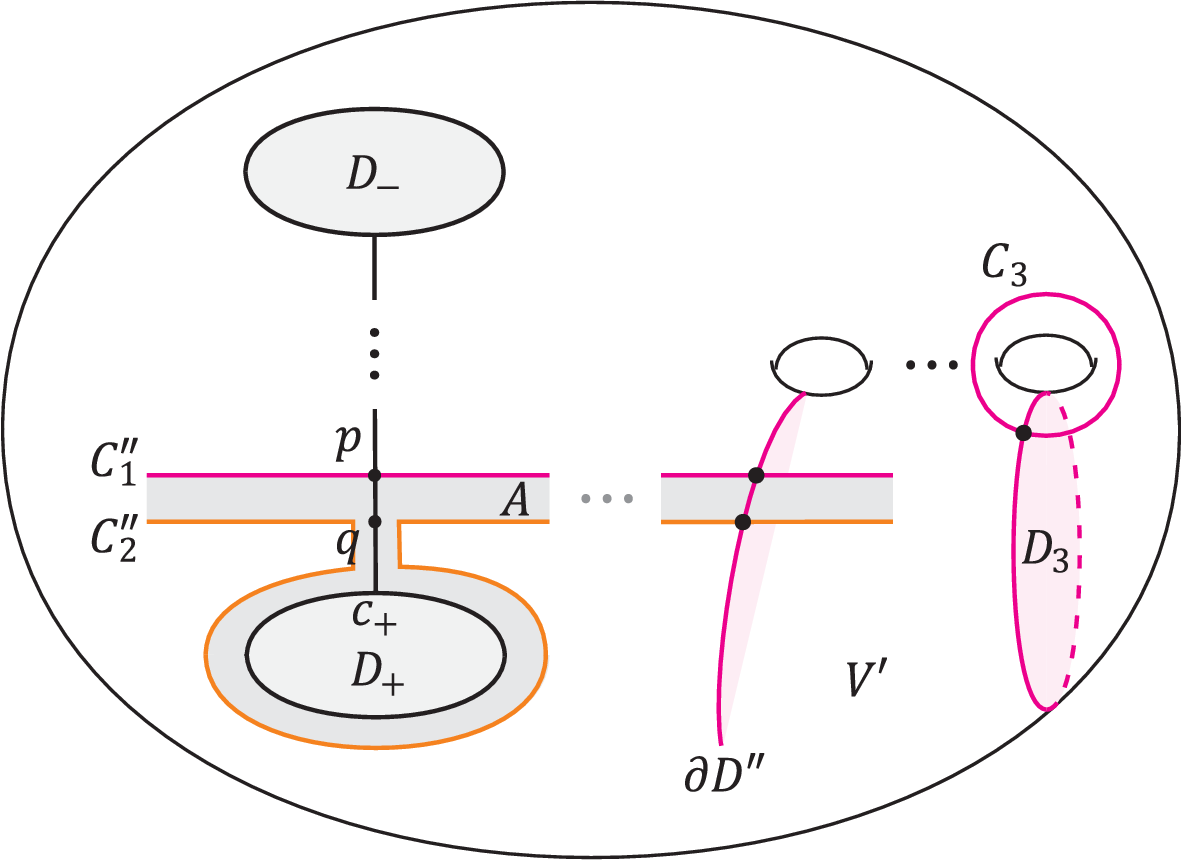}
\caption{The dual pairs $(C''_1, D'')$ and $(C''_2, D'')$.}
\label{fig21}
\end{figure}

Since $(C''_1, D'')$ is a dual pair in a genus-$(g-1)$ handlebody $V'$, and $C''_1 \cup C''_2$ bounds an annulus, and $D_{-}$ is disjoint from $D'' \cup A$, we can take a new dual pair $(C_3, D_3)$ in $V'$ disjoint from $D'' \cup A \cup D_{-}$.
All of $(C''_1, D'')$ and $(C''_2, D'')$ and $(C_3, D_3)$ can be regarded as dual pairs in $V$ because they are disjoint from $D_{+} \cup D_{-}$.
The dual pairs $(C''_1, D'')$ and $(C''_2, D'')$ are p-connected via $(C_3, D_3)$.
Since $| (C''_2 \cup D'') \cap \alpha | = | (C''_1 \cup D'') \cap \alpha | - 1$, by an inductive argument $(C''_2, D'')$ and $(C, D)$ are p-connected.

\vspace{0.3cm}

\noindent Case $2$. $p \in \partial D''$.

Let $D''_1 = D''$, and $D''_2$ be a parallel copy of $D''_1$ such that a point $q \in \partial D''_2 \cap \alpha$ is closer to $c_{+}$ than $p$.
Let $\alpha_{+}$ be the subarc of $\alpha$ from $q$ to $c_{+}$.
Slide a small neighborhood of $q$ in $D''_2$ along $\alpha_{+} \cup D_{+}$.
Then $D''_1 \cup D''_2$ bounds a region $B$ in $V'$ that is homeomorphic to $D'' \times I$ containing $D_{+}$.

Since $(C'', D''_1)$ is a dual pair in a genus-$(g-1)$ handlebody $V'$, and $D''_1 \cup D''_2$ bounds a $3$-ball, and $D_{-}$ is disjoint from $B \cup C''$, we can take a new dual pair $(C_3, D_3)$ in $V'$ disjoint from $B \cup C'' \cup D_{-}$.
All of $(C'', D''_1)$ and $(C'', D''_2)$ and $(C_3, D_3)$ can be regarded as dual pairs in $V$ because they are disjoint from $D_{+} \cup D_{-}$.
The dual pairs $(C'', D''_1)$ and $(C'', D''_2)$ are p-connected via $(C_3, D_3)$.
Since $| (C'' \cup D''_2) \cap \alpha | = | (C'' \cup D''_1) \cap \alpha | - 1$, by an inductive argument $(C'', D''_2)$ and $(C, D)$ are p-connected.
\end{proof}

From second paragraph to the end of the proof of Lemma \ref{lem4.2}, we observe the following lemma (with different notations to avoid confusion).

\begin{lemma}\label{lem4.3}
Let $(C_1, D_1)$ and $(C_2, D_2)$ be dual pairs for $V$.
Suppose that $(C_2 \cup D_2) \cap D_1 = \emptyset$.
Then  $(C_1, D_1)$ and $(C_2, D_2)$ are p-connected.
\end{lemma}

Now we are ready to prove Theorem \ref{thm1.1} in the case of $g \ge 3$.
Actually we show a slightly stronger version of it as follows.

\begin{theorem}\label{thm4.4}
Let $(C, D)$ and $(C', D')$ be dual pairs for a genus-$g$ handlebody $V$ with $g \ge 3$.
Then $(C, D)$ and $(C', D')$ are p-connected.
\end{theorem}

\begin{proof}
As in the proof of Theorem \ref{thm2.4}, we may assume that $D$ and $D'$ are disjoint.
By Lemma \ref{lem2.3}, there exists a sequence $C = C_1, C_2, \ldots, C_n = C''$ of primitive curves with a common dual disk $D$ such that $C_i$ and $C_{i+1}$ are disjoint for each $i \in \{ 1, 2, \ldots, n-1 \}$ and $| C'' \cap \partial D' | \le 1$.
Since $(C_i, D)$ and $(C_{i+1}, D)$ are p-connected for each $i \in \{ 1, 2, \ldots, n-1 \}$ by Lemma \ref{lem4.2}, $(C, D)$ and $(C'', D)$ are p-connected.
See Figure \ref{fig22}.
It remains to show that $(C'', D)$ and $(C', D')$ are p-connected.

\begin{figure}[!hbt]
\includegraphics[width=7cm,clip]{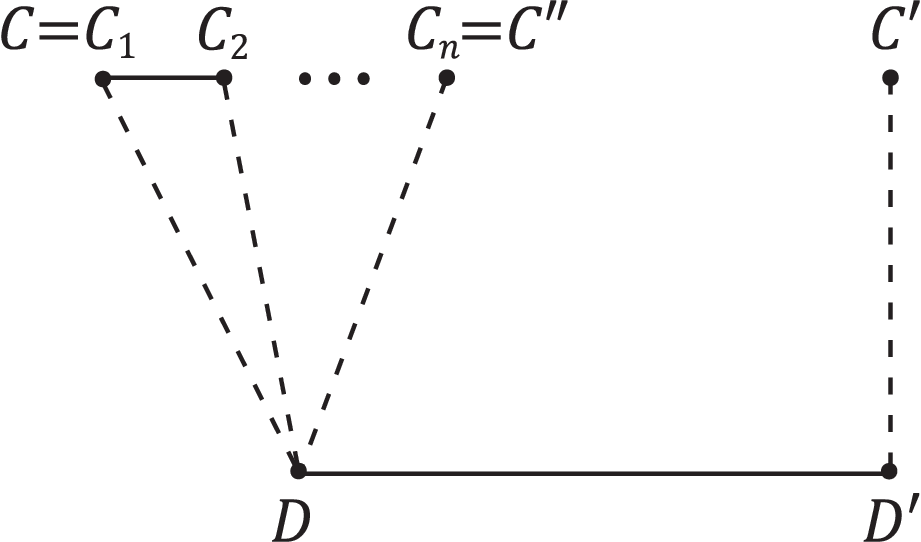}
\caption{$(C, D)$ and $(C'', D)$ are p-connected and $| C'' \cap \partial D' | \le 1$.}
\label{fig22}
\end{figure}

Suppose that $| C'' \cap \partial D' | = 1$.
Then $D'$ is a common dual disk of $C''$ and $C'$.
By Lemma \ref{lem2.2}, there exists a sequence $C'' = C_1, C_2, \ldots, C_k = C'$ of primitive curves (by an abuse of notations) with a common dual disk $D'$ such that $C_i$ and $C_{i+1}$ are disjoint for each $i \in \{ 1, 2, \ldots, k-1 \}$.
Since $(C_i, D')$ and $(C_{i+1}, D')$ are p-connected for each $i \in \{ 1, 2, \ldots, k-1 \}$ by Lemma \ref{lem4.2}, $(C'', D')$ and $(C', D')$ are p-connected.
Since $(C'', D)$ and $(C'', D')$ are p-connected by Lemma \ref{lem4.1}, $(C'', D)$ and $(C', D')$ are p-connected.

Now suppose that $| C'' \cap \partial D' | = 0$.
Then since $(C'' \cup D) \cap D' = \emptyset$, by Lemma \ref{lem4.3} $(C'', D)$ and $(C', D')$ are p-connected.
\end{proof}

\bibliographystyle{amsplain}

\end{document}